\documentclass{amsart}

\usepackage{graphicx}
\usepackage{tikz,tikz-cd}

\newtheorem{theorem}{Theorem}[section]
\newtheorem{lemma}[theorem]{Lemma}
\newtheorem{corollary}[theorem]{Corollary}

\theoremstyle{definition}
\newtheorem{definition}[theorem]{Definition}
\newtheorem{example}[theorem]{Example}

\theoremstyle{remark}
\newtheorem{remark}[theorem]{Remark}

\numberwithin{equation}{section}

\begin{document}
	
	\title{Continuity of Functions on Bare Representation of Graphs under Star Topology}
	
	\author{Rodolfo E. Maza}
	\address{Map\'ua University, Intramuros, Manila 1002, Philippines}
	\email{remaza@mapua.edu.ph}
	\address{University of the Philippines-Diliman, Quezon City 1101, Philippines}
	\email{remaza@up.edu.ph}
	
	
	\begin{abstract}
		This paper introduces a novel topology, referred to as the \textit{star topology}, on finite graphs. By treating vertices and edges as points in a unified space, we explore continuous maps between Bare representations of a graph and their properties. The key distinction lies in the fact that while every graph homomorphism induces a continuous map, the converse is not generally true due to potential loss of adjacency information.
	\end{abstract}
	
	\maketitle
	
	\section{Introduction}
	\label{Intro:sec}
	
	Graphs are fundamental structures in mathematics, computer science, and engineering, serving as building blocks for complex systems. Various approaches exist to associate topology with graphs, each emphasizing different structural aspects—from treating graphs as 1-dimensional CW-complexes in algebraic topology \cite{massey2019basic}, to exploring graph-based spaces using adjacency relations in digital topology \cite{NoglySchladt1996}, or applying Alexandroff topologies to finite graphs leveraging their minimal neighborhood properties \cite{JafarianAmiriJafarzadeh2013}. Topological concepts like continuity and connectedness have also been reinterpreted for directed graphs with applications in group theory \cite{serre1977arbres}.
	
	However, existing frameworks often separate vertices and edges into distinct domains or prioritize adjacency preservation over holistic structural analysis. This paper introduces the \textbf{star topology} on the \textbf{bare representation} \( B(G) \) of a finite graph \( G \), where \( B(G) \) is the disjoint union of vertices \( V(G) \) and edges \( E(G) \). Unlike prior methods, this topology unifies vertices and edges into a single topological space by defining open sets using incidence relations. Specifically, neighborhoods around vertices in the star topology include both the vertex itself and all incident edges.
	
	A crucial distinction arises: while every graph homomorphism induces a continuous map in the star topology, the converse fails because continuity allows edges to map to vertices—a phenomenon incompatible with the conditions for graph homomorphisms. Evenmore, a weak-homomorphism or an egamorphism induces a continuous map while the converse is false.
	
	This paper proceeds as follows: Section \ref{Prelim:sec} introduces the basic definitions and established results. Section \ref{StarTopo:sec} formalizes the star topology and proves foundational properties like Alexandroff compatibility and establishes the equivalence between graph connectivity and topological connectedness of \( B(G) \). Section \ref{ContiMap:sec} examines continuous maps, demonstrating how vertex-identification operations maintain continuity even when violating homomorphism conditions.

	\section{Preliminaries}
	\label{Prelim:sec}
	Let us begin by defining key concepts used in this paper. We build on
	foundational ideas from Graph Theory and Topology as outlined in \cite{ChartrandZhang}
	and \cite{Munkres}.
	
	\begin{definition}\cite{ChartrandZhang}
		A \textit{graph} \( G = (V, E) \) consists of a vertex set \( V \) and an
		edge set \( E \) where \(E\) consists of 2-element subsets of \(V\). These two
		sets are presumed to be disjoint \cite{ChartrandZhang}. Occasionally, we denote by
		\(V(G)\) and \(E(G)\) the vertex and edge sets of \(G\). 
	\end{definition}
	
	\begin{definition}\cite{ChartrandZhang}
		The \textit{degree} of a vertex \(v \in V(G)\), denoted as \(\deg(v)\),
		is the number of edges incident to \(v\). Formally,
		\[
		\deg(v) = |\{e \in E(G): e \text{ is incident to } v\}|.
		\]
	\end{definition}
	
	\begin{definition}\cite{ChartrandZhang}
		A vertex \(u \in V(G)\) is called \textit{isolated} if there are no edges
		incident to it, i.e., \(\deg(u) = 0\).
	\end{definition}
	
	\begin{example}
		Consider a graph \(G\) with vertices \(v_1, v_2, v_3\) and edges \(e_1 =
		v_1v_2\), \(e_2 = v_1v_3\). The degree of \(v_1\) is 2, while the degrees of
		\(v_2\) and \(v_3\) are each 1.
	\end{example}
	
	\begin{definition} \cite{ChartrandZhang}
		Two graphs \(G\) and \(H\) are disjoint if \(V(G)\) and \(V(H)\) are disjoint. Their union \(G + H\)
		is the graph with vertex set \(V(G) \cup V(H)\) and edge set \(E(G) \cup E(H)\).
	\end{definition}
	
	\begin{definition} \cite{ChartrandZhang}
		An subgraph induced by a set \(A\) of vertices of a graph \(G\), denoted by \(G[A]\), is the graph with vertex set \(A\) and edge set consisting of edges in \(G\) with both end-vertices are in \(A\).
	\end{definition}
	
	\begin{definition} \cite{ChartrandZhang}
		A graph \(H\) is a subdivision of \(G\) if \(H\) can be obtained by inserting vertices of
		degree 2 into edges of \(G\).
	\end{definition}
	
	\begin{definition}\cite{Knauer2019algebraic}
		A (strong) graph homomorphism is a map \(f: V(G) \rightarrow V(H)\) such that for
		any edge \( uv \in E(G) \), there exists an edge \( f(u)f(v) \in E(H) \).
	\end{definition}
	
	\begin{definition}\cite{Knauer2019algebraic}
		A weak graph homomorphism or an egamorphism between graphs \( G \) and \( H \) is a map \(f: V(G) \rightarrow V(H)\) such that if \( u v \in E(G) \), then either \(f(u) = f(v)\) or \( f(u)f(v) \in E(H) \).
	\end{definition}
	
	\begin{remark}
		Since our graphs in this paper are all simple, then a graph homomorphism is not equivalent to weak graph homomorphism. \cite{Knauer2019algebraic}
	\end{remark}
	
	\begin{definition}\cite{ChartrandZhang}
		A graph is connected if there exists a path between any pair of vertices.
	\end{definition}
	
	\begin{definition}\cite{Munkres}
		A \textit{topological space} is a set \( X \) equipped with a topology \(
		(\tau) \), which is a collection of subsets of \( X \) satisfying:
		\begin{enumerate}
			\item The union of any collection of sets in \( \tau \) is also
			in \( \tau \).
			\item The intersection of any finite number of sets in \( \tau \)
			is also in \( \tau \).
			\item Both the empty set and \( X \) itself are in \( \tau \).
		\end{enumerate}
	\end{definition}
	
	\begin{definition}\cite{Munkres}
		\label{Closed Sets}
		A subset \( C \subseteq X \) is called \textit{closed} in the topological
		space \( (X, \tau) \) if its complement \( X \setminus C \) is open.
	\end{definition}
	
	\begin{definition}\cite{Munkres}
		A collection \( \mathcal{B} \) of open sets in a topological space \( (X,
		\tau) \) is said to be a \textit{basis} for the topology if
		\begin{enumerate}
			\item every open set in \( \tau \)
			can be written as a union of elements from \( \mathcal{B} \), and
			\item for every \(B_1, B_2 \in \mathcal{B}\), the intersection
			\(B_1 \cap B_2\) can be express as a union of sets in \(\mathcal{B}\).
		\end{enumerate}
	\end{definition}
	
	\begin{definition}\cite{Munkres}
		A collection \( \mathcal{S} \) of subsets of \( X \) is called a
		\textit{sub-basis} for the topology on \( X \) if the topology \( \tau \)
		can be
		generated by taking all possible unions of finite intersections of
		elements from
		\( \mathcal{S} \).
	\end{definition}
	
	\begin{definition} \cite{Munkres}
		A function \(f:X \rightarrow Y\) between topological spaces \(X\) and \(Y\) is said to be continuous if \(f^{-1}(U)\) is open in \(X\) for every open set \(U\) in \(Y\).
	\end{definition}
	
	\begin{remark}
		\label{ContinuityviaSub-basis}
		It suffices to show that the pre-image of every subbase in topology of \( Y \) is open in \( X \) to demonstrate that a function \( f: X \rightarrow
		Y \) between topological spaces is continuous \cite{Munkres}.
	\end{remark}
	
	\begin{definition}[Homeomorphism]
		A continuous map that is bijective and has a continuous inverse is called
		a
		homeomorphism. Such maps establish topological equivalence between
		spaces, as
		detailed in \cite{Munkres}.
	\end{definition}
	
	\begin{definition}[Alexandroff Space]
		An Alexandroff space is a topological space where every point has a
		minimal
		neighborhood, or equivalently, has a unique minimal base. Equivalently, the
		intersection of every family of open sets is itself an open set.
		\cite{Arenas,JafarianAmiriJafarzadeh2013}
	\end{definition}

	\section{Star Topology on Bare Representations}
	\label{StarTopo:sec}
	The \textbf{Bare representation} \( B(G) \) of a graph \( G \) is a set which is a disjoint
	union of \( V(G) \) and \( E(G) \). We prescribe a topology on \( B(G) \)
	using the incidence relation on \( G \).
	
	\begin{definition}[Open Stars]
		The open star centered at a vertex \( v \in V(G) \), denoted by \( S_G(v)
		\), consists of \( v
		\) itself along with all edges incident to \( v \). Formally,
		\[ S_G(v) = \{v\} \cup \{e \in E \mid e \text{ is incident to } v\}. \]
	\end{definition}
	
	\begin{remark}
		Note that the collection \(\{S_G(v): v \in V(G)\}\) forms a sub-basis.
	\end{remark}
	
	\begin{definition}[Star Topology]
		The \textit{star topology} on \( B(G) \) is the topology generated by the
		collection \(\{S_G(v): v \in V(G)\}\).
	\end{definition}
	
	\begin{theorem}[Basic Property]
		\label{BasicProperty}
		We have the following properties of the star topology on \(B(G)\):
		\begin{enumerate}
			\item The collection
			\[
			\{S_G(v): v \in V(G)\} \cup \{\{e\}: e \in E(G)\}
			\]
			forms a basis for the star topology.
			\item For any subset \(A\) of \(B(G)\), \(A \cap E(G)\) is open.
			\item A subset \(A \subseteq B(G)\) is open if and only if
			\(S_G(v) \subseteq A\) for every vertex \(v \in A\).
		\end{enumerate}
	\end{theorem}
	\begin{proof}
		All the statement in Theorem \ref{BasicProperty} follow immediately from
		fact that the singleton set \(\{uv\}\) is open in \( B(G) \) if \(u\) and \(v\)
		are adjacent because it is the intersection of the open stars \( S_G(u)\) and \(
		S_G(v) \).
	\end{proof}
	
	\begin{remark}
		It is easy to see that in any graph \(G\), the representation \(B(H)\) of
		the subgraph \(H\) of \(G\) is subset of \(B(G)\).
	\end{remark}

	\begin{figure}
		\centering
		\begin{tikzpicture}[scale=0.8]
			\foreach \i in {1,...,5} {
				\node[fill=blue] (v\i) at (\i*72+54:2cm) [circle, draw] {};
			};
			\foreach \i in {1,...,4} {
				\node[fill=blue] (u\i) at (\i*72+54:4cm) [circle, draw] {};
			};
			\node[fill=red] (u5) at (5*72+54:4cm) [circle, draw] {};
			\draw[blue] (v1) -- (v3);
			\draw[blue] (v2) -- (v4);
			\draw[blue] (v3) -- (v5);
			\draw[blue] (v4) -- (v1);
			\draw[blue] (v5) -- (v2);
			
			\draw[blue] (u1) -- (u2);
			\draw[blue] (u2) -- (u3);
			\draw[blue] (u3) -- (u4);
			\draw[red] (u4) -- (u5);
			\draw[red] (u5) -- (u1);
			
			\draw[blue] (v1) -- (u1);
			\draw[blue] (v2) -- (u2);
			\draw[blue] (v3) -- (u3);
			\draw[blue] (v4) -- (u4);
			\draw[red] (v5) -- (u5);
		\end{tikzpicture}
		\caption{The Petersen graph with a subgraph (blue) and an open star(red).}
	\end{figure}
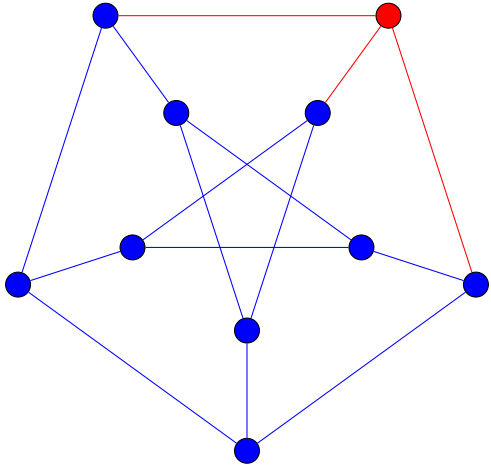
	
	\begin{theorem}[Characterization of Open Sets]
		\label{Characterization of Open Sets}
		A set \(A \subseteq B(G)\) is open if and only if there is a subgraph \(H\) of
		\(G\) such that \(B(H) = B(G) \setminus A\).
	\end{theorem}
	\begin{proof}
		If \(A\) is open, then for every vertex \(v \in A\), the star \(S_G(v)\) is
		contained in \(A\).
		Consider \(U\) as the set of all vertices not in \(A\) and \(F\) as the set of
		all edges not in
		\(A\). Since any edge with one endpoint in \(A\) must entirely lie within \(A\)
		(due to the openness
		of \(A\)), it follows that every edge in \(F\) has both endpoints in \(U\).
		Therefore, we can define
		\(H\) as the subgraph of \(G\) with vertex set \(U\) and edge set \(F\).
		
		Conversely, if there exists a subgraph \(H\) such that \(B(H) = B(G) \setminus
		A\), then \(A\)
		consists of all vertices not in \(H\) and all edges not in \(H\). If \(v \in A\)
		is a vertex, then the edges in \(S_G(v)\) are not in \(B(H)\) as the edges in
		\(B(H)\) must have end-vertices in \(B(H)\). Hence, \(S_G(v)\) is contained in
		\(A\) for all vertices \(v \in A\). By Remark \ref{BasicProperty}, \(A\) is open.
	\end{proof}
	
	\begin{corollary}[Characterization of closed sets as representation of subgraphs]
		The closed sets of \(B(G)\) are precisely those representations of
		subgraphs of \(G\).
	\end{corollary}
	\begin{proof}
		From Definition \ref{Closed Sets}, this is equivalent to Theorem
		\ref{Characterization of Open Sets}.
	\end{proof}
	
	\begin{corollary}[Isolated Vertex Properties]
		\label{IsolatedVertexProperties}
		
		For any vertex \( v \in V(G) \):
		\begin{enumerate}
			\item The singleton set \(\{v\}\) is closed in the star topology.
			\item The singleton set \(\{v\}\) is open in the star topology if
			and only if \( v \) is an isolated vertex.
		\end{enumerate}
	\end{corollary}
	
	\begin{proof}
		A single vertex is subgraph of \(G\). Hence, \(\{v\}\) is closed by
		Theorem \ref{Characterization of Open Sets}. Furthermore, if \(v\) is an isolated
		vertex, then \(\{v\} =  S_G(v)\) is open by definition. If \(\{v\}\) is open,
		\(S_g(v) \subseteq \{v\} \subseteq S_G(v)\) by Theorem \ref{BasicProperty}. Thus,
		\(v\) is not incident to any edges of \(G\) and so, \(v\) is an isolated `vertex.
	\end{proof}
	
	\begin{remark}
		It follows immediately from the definition of the star topology that it induces
		an Alexandroff topology on graphs. In \cite{JafarianAmiriJafarzadeh2013}, an
		Alexandroff topology is defined as a topology on the vertex set of a graph such
		that basic open sets are the open neighborhood of vertices.
		
		Furthermore, in \cite{NoglySchladt1996}, it is stated that for a given graph \( G
		\) and a set \( V \subseteq V(G) \), one can define an induced subgraph \( G_V
		\). The paper posed a question of whether such a subset \( V \) can be assigned
		a topology by declaring certain subsets of \( V \) as ``open'' in a meaningful way.
		In our case, the star topology provides this structure by associating openness
		with the incidence relation and open stars.
	\end{remark}
	
	For every graph \(G\), the star topology on \(B(G)\) satisfies the $T_0$ separation axiom. Indeed,
	the open stars separate the vertices; the singleton set of an edge is an open set separating the
	edge from the rest of \(B(G)\). However, we cannot separate an edge and a vertex incident to it by
	two disjoint open sets. Hence, the star topology is not always Hausdorff. A simple example of this is the graph \(K_2\).
	
	\subsection{Connectedness}
	
	Recall that a graph \( G \) is connected if every pair of vertices are endpoints of some path. Contrasting this with connectivity in topology, a space is connected if it
	cannot be expressed as a disjoint union of open sets. Although this term has been used in two different fields, it should be clear that when we say that a graph is connected, we refer to
	connectedness in graph theory. Similarly, the concept of connectedness for \( B(G) \) is defined within the context of topology.
	
	\begin{theorem}
		A graph \(G\) is connected if and only if \(B(G)\) is connected.
	\end{theorem}
	\begin{proof}
		Assume that \( G \) is a connected graph. We aim to show that its Bare representation
		\(
		B(G) \) under
		the star topology is also connected.
		
		For contradiction, suppose that \( B(G) \) is disconnected. Then, there exist two
		non-empty disjoint
		open sets \( U \) and \( V \) in \( B(G) \) such that:
		\[
		B(G) = U \cup V
		\]
		Choose a vertex \( u \in U \) and another vertex \( v \in V \). Since \( G \) is
		connected, there
		exists a finite path \( P: u = u_1u_2 \cdots u_n = v \).
		
		Consider each edge \( e = u_iu_{i+1} \) in the path. If one endpoint of an edge
		is
		in \( U \), then
		by the definition of open sets in the star topology, the entire neighborhood star
		around that vertex
		(including the edge itself) must lie within \( U \). Similarly, if a vertex is in
		\(
		V \), its
		neighborhood star lies entirely in \( V \).
		
		However, along the path from \( u \) to \( v \), there must be an edge where one
		endpoint
		transitions from \( U \) to \( V \). This transition would imply that the edge
		belongs to both \( U
		\) and \( V \), contradicting their disjointness. Hence, \( B(G) \) is connected.
		
		To establish the converse, we assume that \( B(G) \) is connected under the star
		topology and demonstrate that \( G \) is connected. Assume, for contradiction, that \( G \) is not connected. Then, it can be partitioned into two vertex-disjoint non-empty
		subgraphs \( H_1 \) and \( H_2 \). Also, \( H_1 \) and \( H_2 \) are edge-disjoint.
		
		Since the set \( B(H_i) \) consists of all vertices and edges of \( H_i \), it follows that \( B(H_1) \) and \( B(H_2) \) are non-empty disjoint sets whose union is $B(G)$. In the star topology, any neighborhood star around a vertex	within \( H_1 \) includes only vertices and edges from \( H_1 \), meaning \( B(H_1) \) is an open set. The same applies to \( B(H_2) \). Therefore, we have:
		\[
		B(G) = B(H_1) \cup B(H_2)
		\]
		where both \( B(H_1) \) and \( B(H_2) \) are open, non-empty, and disjoint. This contradicts the assumption that \( B(G) \) is connected. Hence, our initial assumption that \( G \) is disconnected must be false, which shows that \( G \) is connected.
	\end{proof}
	
	\section{Continuous Maps}
	\label{ContiMap:sec}

	Having established the star topology on bare representations, we now explore
	the behavior of continuous functions in this topological setting. Our primary goal in this section is to investigate how continuous maps interact with the graph structure and topology and list some examples of continuous maps. 
	
	\begin{example}
		\label{Homomorphismsarecontinuous}
		We begin by examining a graph homomorphism \( f: G \to H \). Consider the map \(\Phi_f: B(G) \to B(H)\) defined by:
		\begin{itemize}
			\item For each vertex \( v \in V(G) \), \(\Phi_f(v) = f(v)\).
			\item For each edge \( e = uv \in E(G) \), \(\Phi_f(e) = f(u)f(v)\) in \(
			B(H) \).
		\end{itemize}
		It should follow immediately that \(\Phi^{-1}(S_H(\Phi_f(x))) = S_G(x)\) for every vertex \(x\) of \(G\). Hence, graph homomorphisms induce a continuous map between the graphs.
	\end{example}
	
	However, this is not always the case that every continuous map is a function induced by a graph homomorphism. To distinguish these types of mappings, we introduce the following definitions:
	
	\begin{definition}
		A \textbf{vertex map} between bare representations of graphs is a map that sends vertices to vertices. Similarly, an \textbf{edge map} is a function that sends edges to edges. An \textbf{incidence map} is a continuous function that is both a vertex map and an edge map.
	\end{definition}
	
	Since only non-adjacent vertices are the only possible vertices mapped to a single vertex under an incidence map, the following should be immediate from the definition.
	
	\begin{remark}
		\label{IncidenceMapandIndependentSets}
		Let \(f:B(G) \rightarrow B(H)\) be an incidence map. For every \(v \in V(H)\), the \(f^{-1}(v)\) is an independent set. Consequently, every graph homomorphism induces an incidence map, and vice versa.
	\end{remark}
	
	Now consider a weaker notion of graph homomorphism.
	
	\begin{theorem}
		\label{WeakHomomorphismareContinuousVertexMap}
		Let \(f:V(G) \rightarrow V(H)\) be a weak graph homomorphism. Then \(f\) induces a continuous map from \(B(G)\) to \(B(H)\).
	\end{theorem}
	\begin{proof}
		We define the map \(\Phi_f:B(G) \rightarrow B(H)\) by setting \(\Phi_f(u) = f(u)\) for every vertices \(u\) of \(G\) and every edge \(xy\) is mapped to \(f(x)f(y)\) if \(f(x) \neq f(y)\) and \(f(x)\), otherwise.
		
		Let \(S_H(v)\) be a given star in \(H\). We apply Theorem \ref{BasicProperty} to show that \(f^{-1}(S_H(v))\) is open and this is sufficient to prove the theorem.
		
		Let \(x \in \Phi_f^{-1}(S_H(v))\) be a vertex. Consider an edge \(xy\) incident to \(x\). If \(f(x) = f(y)\), then \(f(xy) = f(x)\) by definition. If \(f(x) \neq f(y)\) then \(f(x)\) is adjacent to \(f(y)\) in \(H\) and so, xy is mapped to \(f(x)f(y)\). Either way, \(xy \in \Phi_f^{-1}(S_H(v))\). Hence, \(\Phi_f^{-1}(S_H(v))\) is the union of the stars \(S_G(x)\) such that \(\Phi_f(x) = v\). Therefore, \(\Phi_f^{-1}(S_H(v))\) is open and so, \(\Phi_f\) is continuous.
	\end{proof}
	
	\begin{example}
		While every homomorphism induces a continuous map, the converse is not generally true. For instance, the complete graph \( K_2 \) with vertices \(\{v_1,v_2\}\) and edge \( e =	v_1 v_2 \) and the bare representation of \( B(K_2) \) consists of two vertices and one edge. 
		
		Any continuous map from \( B(K_2) \) to an isolated vertex \( N_1 \) must send both the vertices and the edge to the single vertex in \( B(N_1) \)--an edge-contraction. However, this does not preserve adjacency since the edge in \( K_2 \) would be mapped to the vertex in \( S_1 \).
	\end{example}
	
	Thus, edge contraction provides insight into what occurs during continuous mappings. To generalize this concept, we formally introduced the following.
	
	\begin{definition}[Vertex-Identification]
		\label{Vertex-Identification}
		Let \(G\) be a finite simple graph and let \(u\) and \(v\) be distinct
		vertices
		of \(G\). A \textit{vertex-identification} of \(u\) and \(v\) into a new
		vertex
		\(w\) yields a graph \(G/\{u,v\}\) defined as follows:
		\begin{itemize}
			\item \textbf{Vertex set:}
			\[
			V(G/\{u,v\}) = (V(G) \setminus \{u,v\}) \cup \{w\}.
			\]
			\item \textbf{Edge set:}
			The edge set \(E(G/\{u,v\})\) is the union of:
			\begin{enumerate}
				\item Edges not incident to \(u\) or \(v\):
				\[
				\{e \in E(G) : e \text{ is not incident to } u \text{ or
				} v\}.
				\]
				\item Edges incident to \(u\) or \(v\), redirected to
				\(w\):
				\[
				\{wx \mid x \in V(G) \setminus \{u,v\} \text{ and } (ux
				\in E(G) \text{
					or } vx \in E(G))\}.
				\]
			\end{enumerate}
		\end{itemize}
		If \(u\) and \(v\) are adjacent (i.e., \(uv \in E(G)\)), this process is
		specifically called an \textit{edge contraction}, denoted \(G/e\) or \(G/uv\) (see for example, \cite{ChartrandZhang}).
		
		The vertex-identification map \(\phi_{u,v}: B(G) \rightarrow B(G/\{u,v\})\) is
		defined by:
		\begin{enumerate}
			\item Vertices: All vertices in \(V(G) \setminus \{u,v\}\) are
			mapped to
			themselves. Vertices \(u\) and \(v\) are mapped to \(w\).
			\item Edges:
			\begin{itemize}
				\item Edges \(e \in E(G)\) not incident to \(u\) or \(v\)
				are mapped to
				themselves.
				\item Edges \(ux\) or \(vx\) in \(G\) (where \(x \neq u,
				v\)) are mapped
				to \(wx\) in \(G/\{u,v\}\).
				\item If \(uv\) is an edge (in the case of edge
				contraction), it is
				removed from \(E(G)\) and not included in
				\(E(G/\{u,v\})\). Instead, \(uv\) is
				"collapsed" into \(w\).
			\end{itemize}
		\end{enumerate}
		In the case of edge contraction (\(u\) and \(v\) adjacent), the map \(\phi_{u,v}\) is called an \textit{edge contraction map}. A \textit{contraction map} between bare representations is either the identity or a finite sequence of edge contraction maps 
	\end{definition}
	
	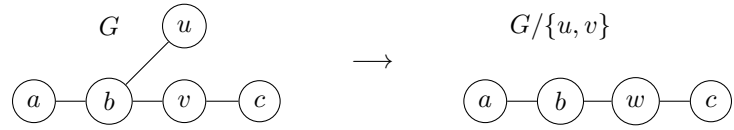
\begin{figure}
		\centering
		\begin{tikzpicture}
			\node (v1) at (0,0) [circle, draw] {\(a\)};
			\node (v2) at (1,0) [circle, draw] {\(b\)};
			\node (v3) at (2,0) [circle, draw] {\(v\)};
			\node (v4) at (3,0) [circle, draw] {\(c\)};
			\node (w) at (2,1) [circle, draw] {\(u\)};
			
			\draw (v1) -- (v2);
			\draw (v2) -- (v3);
			\draw (v3) -- (v4);
			\draw (v2) -- (w);
			
			\node at (1,1) {\( G \)};
			
			\node at (4.5,0.5) {\( \longrightarrow \)};
			
			\node (u1) at (6,0) [circle, draw] {\(a\)};
			\node (u2) at (7,0) [circle, draw] {\(b\)};
			\node (u3) at (8,0) [circle, draw] {\(w\)};
			\node (u4) at (9,0) [circle, draw] {\(c\)};
			
			\draw (u1) -- (u2);
			\draw (u2) -- (u3);
			\draw (u3) -- (u4);
			
			\node at (7,1) {\( G/\{u,v\} \)};
		\end{tikzpicture}
		\caption{The graph \(G/\{u,v\}\) obtained by identifying the vertices \(u\) and \(v\) into \(w\) while we map the edge \(ub\) and \(vb\) to \(wb\).}
	\end{figure}
	
	\begin{theorem}
		\label{VertexIdentificationisContinuous}
		The vertex-identification map between graphs induces a surjective and continuous map on the corresponding bare representation.
	\end{theorem}
	\begin{proof}
		Let \(\phi_{u,v}:B(G) \rightarrow B(G / \{u,v\})\) be a vertex-identification of the
		vertices \(u\) and \(v\) to a vertex \(w\) of a graph \(G\). Let \(O\) be open in
		\(B(G/\{u,v\})\). In light of Remark \ref{ContinuityviaSub-basis}, we may assume
		that \(O=S_v(x)\) for some vertex \(x \in V(G / \{u,v\})\).
		\begin{enumerate}
			\item Assume \(x \neq w\). Then \(\phi(x) =x\). Now, every edges incident to \(x\) in \(G\) are mapped to themselves in \(G/\{u,v\}\). So, \(S_G(x)\) is contained to \( \phi^{-1}(S_{G / \{u,v\}}(x))\).	Now, the only vertex mapped to \(x\) is \(x\) itself. Also, the edges incident to \(x\) in \(B(H)\) but not to \(w\) are edges \(G\) while the edges \(wx\) has a preimage of the form \(xv\) or \(xu\) in \(G\). Hence,
			\[\phi^{-1}(S_{G / \{u,v\}}(x)) = S_G(x).\]
			\item If \(x = w\), then every edges incident to \(w\) in
			\(G/\{u,v\}\) came from an edge incident to either \(u\) or \(v\) in \(G\). Thus,
			\[\phi^{-1}(S_{G / \{u,v\}}(w)) = S_G(u) \cup S_G(v).\]
		\end{enumerate}
		In either cases, \(\phi^{-1}(O)\) is open. Therefore, \(\phi\) is
		continuous.
	\end{proof}
	
	\begin{corollary}[Continuous surjection]
		Let \( G \) and \( H \) be graphs where \( H \) is obtained from \( G \) by a	sequence of vertex-identification. Then, there exists a surjective continuous map \(f: B(G) \to B(H)\).
	\end{corollary}
	
	\begin{proof}
		By Theorem \ref{VertexIdentificationisContinuous}, vertex-identification maps are surjective and continuous. The composition of a sequence of vertex-identification maps remains continuous and surjective.
	\end{proof}
	
	The contraction induces a continuous map \(\phi_{u,v}: B(G) \to B(G/e)\) where \(e=uv\) and this map sends both vertices and the edge to the single vertex in \( B(G/e) \). Hence, \(\phi\) does not preserve adjacency information, it fails to be a homomorphism. In contrast to vertex-identification maps that are not edge-contraction, the map \(\phi_{u,v}: B(G) \to B(G/\{u,v\})\) preserves adjacency	.
	
	\begin{corollary}
		If \(G\) is a subdivision of \(H\) then there exists a surjective continuous map from \(G\) to \(H\).
	\end{corollary}
	\begin{proof}
		If \(H_1\) is obtained from \(H\) by inserting a vertex \(w\) to an edge \(e=uv\), then we can contract the edge \(wv\) inducing a continuous map from \(H_1\) to \(H\). Since \(G\) is a subdivision of \(H\), there is a sequence of edge contractions from \(G\) to \(H\) that induces surjective continuous map from \(G\) to \(H\).
	\end{proof}
	
	\begin{remark}
		Given a subdivision \(G\) of a graph \(H\) by inserting the vertex \(w\) on the edge \(uv\) of \(H\), we may define a continuous surjective map different from edge contractions as follows: every vertices not adjacent to \(w\) in \(G\) are mapped to themselves in \(H\), every edges not incident to \(w\) are mapped to themselves. The vertex \(w\) and the edges \(uw\) and \(vw\) are all sent to the edge \(uv\).
	\end{remark}
	
	\begin{remark}
		It should be curious to know which maps are continuous. One thing is sure, a continuous map must map a connected Bare representations to a connected Bare representations. So, in most cases, the continuity assumption allows us to prove results in connected graphs and generalized easily to disconnected ones.
	\end{remark}
	
	From the results above, we notice a very simple property of continuous maps.
	
	\begin{theorem}
		\label{ContinuityPreservesIncidence}
		Let \(f:B(G) \rightarrow B(H)\) be a continuous map. If \(v \in V(G)\) and \(e \in E(G)\) are incident in \(G\), then either \(f(v) = f(e)\), or \(f(v)\) is a vertex and \(f(e)\) is an edge incident to \(f(v)\) in \(H\).
	\end{theorem}
	\begin{proof}
		Let \( v \in V(G) \) and \( e \in E(G) \) be incident in \( G \). Since \( f \) is continuous, preimages of open sets are open. If \( f(v) \) is a vertex, then \( S_H(f(v)) \) is open in \(B(H)\), so its preimage contains \(S_G(v)\). Now, \(S_G(v)\) contains \( e \) and so, \( f(e) \in S_H(f(v)) \). Hence, either \( f(e) = f(v) \) or \( f(e) \) is an edge incident to \( f(v) \). If \( f(v) \) is an edge, its singleton set \(\{f(v)\}\) is open. Thus, \(f(S_G(v)) = \{f(v)\}\), forcing \( f(e) = f(v) \).
	\end{proof}
	
	Theorem \ref{ContinuityPreservesIncidence} demonstrates that continuity ensures the preservation of incidence relations, either by mapping vertices and edges to the same point or maintaining their adjacency in the target space. Moreover, we can use this result to guarantee the existence of a continuous vertex map from a continuous map.
	
	\begin{theorem}
		\label{ContinuousmapsmeansContinuousVertexmap}
		If there is a continuous \(f:B(G) \rightarrow B(H)\) then there is a continuous vertex map \(\tilde{f}:B(G) \rightarrow B(H)\).
	\end{theorem}
	\begin{proof}
		We may assume that \(G\) is connected. We construct \(\tilde{f}\) from \(f\). Let \(E_0\) be the set of all the edges of \(H\) that has a vertex in their pre-image. For each \(e \in E_0\), let \(U_e\) be the vertices mapped to \(e\). Let \(\mathcal{S}\) be the union of all the (disjoint) open stars with center in \(E_0\).
		
		By Theorem \ref{ContinuityPreservesIncidence} and connectedness of \(G\), the set of vertices of \(G\) outside \(U_e\) that are adjacent to some vertex in \(U_e\) must be mapped to either one of the end-vertices \(x_e\) and \(y_e\) of \(e\). Let \(X_e\) and \(Y_e\) be the sets of vertices in \(G\) outside of \(U_e\) that are adjacent to some vertex in \(U_e\) that are mapped to \(x_e\) and \(y_e\), respectively.

		We define \(\tilde{f}\) as follows:
		
		\begin{enumerate}
			\item For every point \(p \in B(G) \setminus \mathcal{S}\), \(\tilde{f}(p) = f(p).\)
			\item For every open star \(S_G(u) \subseteq \mathcal{S}\), if \(X_e\) and \(Y_e\) are both non-empty where \(f(u) = e\) then we may arbitrarily choose \(\tilde{f}(p) = y_e \; \text{for all} \; p \in S_G(u).\)
			\item For every open star \(S_G(u) \subseteq \mathcal{S}\), if either \(X_e\) or \(Y_e\) is empty where \(f(u) = e\) then
			\[\tilde{f}(p) = x_e \; \text{for all} \; p \in S_G(u) \; if X_e = \emptyset,\]
			otherwise,
			\[\tilde{f}(p) = y_e \; \text{for all} \; p \in S_G(u) \; if Y_e = \emptyset.\]
		\end{enumerate}
	\end{proof}
	
	We establish Theorem \ref{injective continuous maps are homomorphism}, which shows that under certain conditions, injective continuous maps correspond precisely to graph homomorphisms.
	
	\begin{theorem}
		\label{injective continuous maps are homomorphism}
		Let \(G\) and \(H\) be graphs. Suppose \(f: B(G) \rightarrow B(H)\) is an injective continuous map. Then, the restriction
		\[
		f|_{V(G)}: V(G) \rightarrow V(H)
		\]
		is an injective graph homomorphism.
	\end{theorem}
	\begin{proof}
		Let \(u\) be a vertex in \(G\). Since \(G\) has no isolated vertex, there is an edge \(e\) incident to \(v\) in \(G\). Since \(f\) is injective, \(f(u)\) and \(f(e)\) are distinct in \(B(H)\). Since \(f\) is continuous, then \(f(u)\) is a vertex and \(f(e)\) is an edge incident to \(f(u)\) in \(B(H)\) by Theorem \ref{ContinuityPreservesIncidence}. Hence, \(f\) maps vertices to vertices and so, the restriction
		\[f|_{V(G)}:V(G) \rightarrow V(H)\]
		is a well-defined injective function. It remains to show that \(f\) is a homomorphism. Let \(u\) and \(v\) be adjacent in \(G\) and that \(e=uv \in E(G)\). From what we have shown above, \(f(u)\) and \(f(v)\) are vertices both incident to the edge \(f(e)\) in \(B(H)\). It follows that \(f(u)\) and \(f(v)\) are adjacent in \(B(H)\). Therefore, \(f|_{V(G)}\) is a homomorphism of graphs.
	\end{proof}
	
	\begin{corollary}
		A homeomorphism between the Bare representations of two graphs \(G\) and \(H\) induces a graph isomorphism from \(G\) to \(H\).
	\end{corollary}
	
	\begin{proof}
		Let \(f: B(G) \to B(H)\) be a homeomorphism. By definition, \(f\) is a bijective continuous map with a continuous inverse. If \(G\) is disconnected, then the continuity of \(f\) and its inverse maps the components of \(G\) in one-to-one correspondence to the components of \(H\). In particular, isolated vertices in \(G\) must be mapped to isolated vertices in \(H\). Hence, we may assume that \(G\) and \(H\) has no isolated vertices.
		
		Since \(f\) is injective and continuous, by Theorem \ref{injective continuous maps are homomorphism}, the restriction \(f|_{V(G)}: V(G) \to V(H)\) is a graph homomorphism. The inverse map \(f^{-1}\) is also continuous and injective. Applying Theorem \ref{injective continuous maps are homomorphism} to \(f^{-1}\), its restriction \(f^{-1}|_{V(H)}: V(H) \to V(G)\) is a graph homomorphism. Thus, a homeomorphism between \(B(G)\) and \(B(H)\) induces a graph isomorphism from \(G\) to \(H\).
	\end{proof}
	
	On the other hand, we can show that contraction maps send connected induced subgraphs to a vertex.
	
	\begin{theorem}
		\label{ContractionPreimageConnectedSets}
		Let \(f:B(G) \rightarrow B(H)\) be a contraction map. For every \(v \in f(V(G))\), there is a set \(S\) of vertices of \(G\) such that
		\[f^{-1}(v) = B(G[S]).\]
		Moreover, \(G[S]\) (or \(B(G[S])\)) is connected.
	\end{theorem}
	\begin{proof}
		Let \(S\) to be the set of all vertices in \(G\) mapped to \(v\). Let \(xy \in f^{-1}(v)\) be an edge in \(G\). Then the end-vertices \(x\) and \(y\) are both mapped to \(v\) by Theorem \ref{ContinuityPreservesIncidence}. It follows that \(f^{-1}(v)\) consists of \(S\) and edges with end-vertices in \(S\). But these edges are just the edges of \(G[S]\). Hence,
		\[f^{-1}(v) = B(G[S]).\]
		Furthermore, \(f^{-1}(v)\) can not be disconnected, otherwise, \(\{v\}\) is disconnected.
	\end{proof}
	
	Building on Theorem \ref{ContinuityPreservesIncidence} and \ref{ContinuousmapsmeansContinuousVertexmap}, we see that continuous vertex maps will almost always exists  and therefore, their properties must be carefully examined. Our earlier results establish that all continuous maps between graphs must preserve relations. In the following result, we show that continuous vertex maps can be decomposed into contractions and incidence mappings.
	
	\begin{lemma}
		\label{ContinuousVertexMapsFactorization}
		Let \(G\) and \(H\) be graphs and let \(f:B(G) \rightarrow B(H)\) be a continuous vertex map. If \(u\) and \(v\) are distinct vertices in \(G\), then there is a continuous vertex map \(h:B(G/\{u,v\}) \rightarrow B(H)\) such that following diagram commute:
		\[
		\begin{tikzcd}[column sep=large, row sep=large]
			B(G) \arrow[r, "\phi_{u,v}"] \arrow[rd, "f"'] & B(G/\{u,v\}) \arrow[d, "h"] \\
			& B(H)
		\end{tikzcd}
		\]
	\end{lemma}
	\begin{proof}
		Let \(f:B(G) \rightarrow B(H)\) a vertex continuous map and \(w \in V(H)\) such that there is at least two distinct vertices \(u\) and \(v\) with \(f(u) = f(v) = w\). From the vertex-identification map
		\[\phi_{u,v}: B(G) \rightarrow B(G/\{u,v\})\]
		where \(u\) and \(v\) are identified to new vertex \(z\) in \(G/\{u,v\}\), we define a function \(g:B(G/\{u,v\}) \rightarrow B(H)\) as follows:
		\begin{enumerate}
			\item \(h(z) = w\),
			\item \(h(x) = f(x)\) if \(x \in V(G/\{u,v\})\) with \(x \neq z\),
			\item \(h(x) = f(x)\) if \(x \in E(G/\{u,v\})\) not incident to \(z\), and
			\item \(h(xz) = f(xu)=f(xv)\) if \(xz \in E(G/\{u,v\})\).
		\end{enumerate}
		We only need to verify \(f(xu)=f(xv)\) to insure that \(h\) is well-defined. By assumption, \(f(x)\) is a vertex. By Theorem \ref{ContinuityPreservesIncidence}, there are two possible scenarios: \begin{enumerate}
			\item \(f(xu)\) and \(f(xv)\) are edges in \(H\) incident to both of \(f(x)\) and \(f(z)\) in which case, there is only one such edge in \(H\).
			\item \(f(xu)\) and \(f(xv)\) are vertices in \(H\) such that
			\[f(xv) = f(v) = w = f(u) = f(xv).\]
		\end{enumerate}
		
		Next, \(h\) is a vertex map by construction. Also, \(h\) preserves incidence relations of vertices and edges and so, the image of open stars in \(B(G/\{u,v\})\) under \(h\) are subsets of open stars in \(B(H)\). Hence, \(h^{-1}(S_{H}(y))\) is open in \(G/\{u,v\}\) for all \(y \in V(H)\). By Remark \ref{ContinuityviaSub-basis}, \(h\) is continuous.
		
		Lastly, we verify that the diagram below commutes:
		\[
		\begin{tikzcd}[column sep=large, row sep=large]
			B(G) \arrow[r, "\phi_{u,v}"] \arrow[rd, "f"'] & B(G/\{u,v\}) \arrow[d, "h"] \\
			& B(H)
		\end{tikzcd}
		\]
		Let \(y \in B(G)\). If \(y\) is a vertex in \(V(G) \setminus \{u,v\}\), then:
		\[
		h(\phi_{u,v}(y)) = h(y) = f(y).
		\]
		If \(y = u\) or \(y = v\), then:
		\[
		h(\phi_{u,v}(y)) = h(z) = w = f(u) = f(v).
		\]
		If \(y \in E(G)\) is not incident to \(u\) or \(v\), then \(\phi_{u,v}(y) = y\) is not incident to \(z\) in \(G/\{u,v\}\). Thus,
		\[
		h(\phi_{u,v}(y)) = h(y) = f(y).
		\]
		If \(y\) is incident to \(u\) or \(v\) in \(G\), say \(y=xu\) for some \(x \in V(G)\), then \(\phi_{u,v}(y) = xz\). Hence,
		\[h(\phi_{u,v}(y)) = h(xz) = f(y).\]
	\end{proof}
	
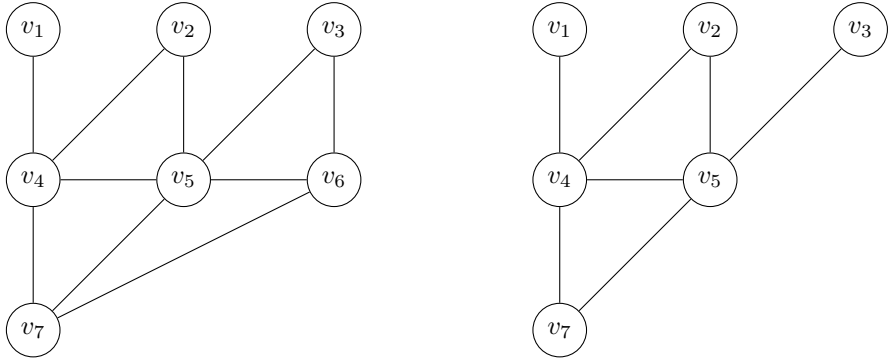
	\begin{tikzpicture}
		\foreach \i in {1,...,7} {
			\node (v\i) at ({mod(\i-1,3)*2},{-(floor((\i-1)/3))*2}) [circle, draw] {\(v_\i\)};
		}
		\draw (v7) -- (v4);
		\draw (v7) -- (v5);
		\draw (v7) -- (v6);
		\draw (v4) -- (v1);
		\draw (v4) -- (v2);
		\draw (v4) -- (v5);
		\draw (v5) -- (v2);
		\draw (v5) -- (v3);
		\draw (v5) -- (v6);
		\draw (v6) -- (v3);

		\foreach \i in {1,2,3,4,5,7} {
			\node (u\i) at ({mod(\i-1,3)*2 +7},{-(floor((\i-1)/3))*2}) [circle, draw] {\(v_\i\)};
		}

		\draw (u7) -- (u4);
		\draw (u7) -- (u5);
		\draw (u4) -- (u1);
		\draw (u4) -- (u2);
		\draw (u4) -- (u5);
		\draw (u5) -- (u2);
		\draw (u5) -- (u3);
		
	\end{tikzpicture}
	
	\begin{theorem}
		\label{FactorizationofContinuousVertexMaps}
		Let \(G\) and \(H\) be graphs where \(G\) is connected and let
		\[f:B(G) \rightarrow B(H)\]
		be a continuous vertex map. Then there exists a graph \(L\), and continuous vertex maps \(g: B(G) \to B(L)\) and \(h: B(L) \to B(H)\), such that:
		\begin{enumerate}
			\item \(g\) is a contraction map.
			\item \(h\) is an incidence map.
			\item The following diagram commutes:
			\[
			\begin{tikzcd}
				B(G) \arrow[r,"g"] \arrow[rd, "f"] & B(L) \arrow[d, "h"] \\
				& B(H)
			\end{tikzcd}
			\]
		\end{enumerate}
		Moreover, there exists a graph \(L'\), and continuous vertex maps \(g': B(G) \to B(L)\) and \(h': B(L') \to B(H)\), such that:
		\begin{enumerate}
			\item \(g'\) is an incidence map.
			\item \(h'\) is a contraction map.
			\item The following diagram commutes:
			\[
			\begin{tikzcd}
				B(G) \arrow[r,"g'"] \arrow[rd, "f"] & B(L) \arrow[d, "h'"] \\
				& B(H)
			\end{tikzcd}
			\]
		\end{enumerate}
	\end{theorem}
	\begin{proof}
		We only prove the first part. We will use induction on the number of vertices in \(G\). The base case is trivial: if \(G\) has only one vertex, then any continuous vertex map from \(B(G)\) to \(B(H)\) must be a constant map, and we can take \(L = G\) with \(g\) and \(h\) being the identity maps.
		
		Assume now that the theorem holds for all graphs with fewer than \(n\) vertices. Let \(G\) have \(n\) vertices. If \(f\) is a incidence map, then we may take \(L=G\), \(g\) the identity map and \(h=f\). Hence, we assume that there are two adjacent vertices \(u\) and \(v\) in \(G\) such that \(f(u)=f(v)\). By Lemma \ref{ContinuousVertexMapsFactorization}, there is a continuous vertex map \(f':B(G/\{u,v\}) \rightarrow B(H)\) such that the diagram below commutes:
		\[
		\begin{tikzcd}[column sep=large, row sep=large]
			B(G) \arrow[r, "\phi_{u,v}"] \arrow[rd, "f"'] & B(G/\{u,v\}) \arrow[d, "f'"] \\
			& B(H)
		\end{tikzcd}
		\]
		But \(G/\{u,v\}\) is order \(n-1\). By induction, there is a graph \(L\) and continuous maps \(g: B(G/\{u,v\}) \to B(L)\) and \(h: B(L) \to B(H)\), such that:
		\begin{enumerate}
			\item \(g\) is a contraction map.
			\item \(h\) is an incidence map.
			\item The following diagram commutes:
			\[
			\begin{tikzcd}
				B(G/\{u,v\}) \arrow[r,"g"] \arrow[rd, "f'"] & B(L) \arrow[d, "h"] \\
				& B(H)
			\end{tikzcd}
			\]
		\end{enumerate}
		Then \(g \circ \phi_{u,v}\) is a contraction map such that the diagram below is a commutative.
		\[
		\begin{tikzcd}[column sep=large, row sep=large]
			B(G) \arrow[r, "g \circ \phi_{u,v}"] \arrow[rd, "f"'] & B(L) \arrow[d, "h"] \\
			& B(H)
		\end{tikzcd}
		\]
	\end{proof}
	
	An immediate consequence of Remark \ref{IncidenceMapandIndependentSets}, Corollary \ref{ContractionPreimageConnectedSets} and Theorem \ref{FactorizationofContinuousVertexMaps} is a result describing continuous functions via incident maps and contraction maps.
	
	\begin{corollary}
		\label{ContinuousVertexMapsPreimage}
		Let \(f: B(G) \rightarrow B(H)\) be a continuous vertex map and let \(v\) be a vertex of \(G\). Then there exists a collection \(\{G_1,G_2,\ldots,G_k\}\) of connected induced subgraphs of \(G\) all mapped to \(v\) by \(f\) such that no edges of \(G\) is connecting a vertex in \(G_i\) and a vertex in \(G_j\) for all \(1 \leq i < j \neq k\).
		
		Moreover, there exists a collection \(\{S_1,S_2,\ldots,S_l\}\) of independent vertex sets of vertices all mapped to \(v\) by \(f\) such that a vertex in \(S_i\) is adjacent to some vertex in \(S_j\) for all \(1 \leq i < j \neq l\).
		
	\end{corollary}
	
	For the next two results, we distinguish the properties of the decompositions describe in Theorem \ref{FactorizationofContinuousVertexMaps}.
	
	\begin{corollary}
		Let \(f: B(G) \rightarrow B(H)\) a continuous vertex map. Then the graph \(L\) in Theorem \ref{FactorizationofContinuousVertexMaps} is unique up to graph isomorphism when \(g\) is a contraction and \(h\) is an incidence map. 
	\end{corollary}
	\begin{proof}
		By Corollary \ref{ContinuousVertexMapsPreimage}, the preimage of a vertex \(v\) of \(H\) under \(f\) consists of connected induced subgraphs of \(G\), say \(G_1,G_2, \ldots, G_k\), such that no edges connects an edge from a vertex in \(G_i\) to a vertex in \(G_j\) for \(1 \leq i<j \leq k\).
		
		Suppose that \(L\) is a graph satisfying Theorem \ref{FactorizationofContinuousVertexMaps} when \(g\) is a contraction and \(h\) is an incidence map. For each \(1 \leq i \leq k\), the graph \(G_i\) has to collapse to a single vertex \(x_i\) under \(g\) where \(x_1,x_2,\ldots, x_k\) are all distinct vertices of \(L\) forming an independent vertex set. Hence, the graphs \(G_1,G_2, \ldots, G_k\) determines uniquely up to graph isomorphism, the graph \(L\).
	\end{proof}
	
	Lastly, the next result provides us with a characterization of weak graph homomoprhism.
	
	\begin{theorem}
		\label{ContinuousVertexMapsareWeakHomomorphisms}
		A vertex map between the bare representations of two graphs is continuous if and only if its restriction
		to the vertices is a weak graph homomorphism.
	\end{theorem}
	\begin{proof}
		Let \(f:B(G) \rightarrow B(H)\) be a continuous vertex map. By Theorem \ref{FactorizationofContinuousVertexMaps}, \(f\) is a composition of a contraction map and an incidence map. Hence, if \(uv\) is an edge in \(G\) then either \(f(u)=f(v)\) or \(f(u)\) and \(f(v)\) are adjacent in \(H\). In other words, the restriction of \(f\) is a weak homomorphism.
		
		The converse is given by Theorem \ref{WeakHomomorphismareContinuousVertexMap}.
	\end{proof}
	
	\section{Continuous surjective maps and related problems}
	
	This section explores problems in graph theory that can interpreted in terms of or related to surjective maps.
	
	\subsection{Walks and Cycles}
	
	Given a closed walk \(W\) of length \(n\) in a graph \(G\) that contains all the edges of \(G\), there is a graph homomorphism from \(C_n\) to \(G\). This homomorphism is surjective and so, the induced continuous vertex map from \(B(C_n)\) to \(B(G)\) has to be surjective.
	
	Questions that may arise are the following:
	\begin{enumerate}
		\item Given a graph \(G\), what is the smallest natural number \(m\) such that there is a continuous surjective map from \(B(C_m)\) to \(B(G)\)?
		\item Given a graph \(G\), what is the largest integer \(n\) such that there is a continuous surjective map from \(B(G)\) to \(B(K_n)\)?
	\end{enumerate}
	The first question is the same as determining the shortest closed walk on a graph \(G\) or the Chinese Postman Problem \cite{edmonds1973matching,guan1962graphic}.
	
	For the second problem, we can bound the size of \(n\). Given a graph \(G\), suppose there is a continuous surjective map \(f\) from \(B(G)\) to \(B(K_n)\). By Theorem \ref{ContinuousmapsmeansContinuousVertexmap}, we may assume that \(f\) is a vertex map. Then surjectivity requires the size of \(G\) to be not less than the size of \(K_n\). In other words, there is a largest integer \(n\) such that there is a continuous map from \(B(G)\) to \(B(K_n)\) and we denote this by \(\theta(G)\). In the next section, we relate \(\theta(G)\) to the chromatic number of a graph.
	
	\subsection{Coloring of graphs}
	
	We may rephrase the definition of coloring in graphs in the context of continuous maps between bare representations. In this case, a coloring of a graph \(G\) is an incidence map from \(B(G)\) to \(B(K_n)\) such that the image of the vertex set of \(G\) is precisely the the vertex set of \(K_n\) for some natural number \(n\) and we say that \(G\) is \(n\) colorable. The chromatic number \(\chi(G)\) is the smallest \(n\) for which \(G\) is colorable. Then the chromatic number \( \chi(G) \) can be characterized as the minimal integer \( n \) such that there exists a surjective incidence map from \( B(G) \) to \( B(K_n) \).
	
	\begin{theorem}
		\label{ColoringintermsofVertexMaps}
		Let \(G\) be an \(n\)-colorable graph. Then \(\chi(G) = n\) if and only if every incidence map from \(B(G)\) to \(B(K_n)\) is surjective.
	\end{theorem}
	\begin{proof}
		Let \(\chi(G) = n\). Suppose \(f:B(G) \rightarrow B(K_n)\) is an incidence map but \(f\) is not surjective. If there is \(v \in V(K_n)\) has no pre-image in \(G\), since \(f\) is an incidence map, no vertex of \(G\) is mapped to \(v\) by \(f\) and also, no edges were mapped to the edges incident to \(v\). Hence, the image of \(G\) under \(f\) is a complete graph with size at most \(n-1\). We can use this to color \(G\) and therefore contradicts the minimality of \(n\). If there is an edge \(ab\) in \(K_n\) with \(a,b \in V(K_n)\) that has no pre-image in \(B(G)\), then we can identify \(a\) and \(b\) so that we have an incidence map from \(B(G)\) to \(B(K_{n-1})\). Again, contradicting the minimality of \(n\). Thus, \(f\) must be surjective.
		
		Conversely, it suffices to show that \(n\) is the smallest natural number for which \(G\) is \(n\)-colorable. Suppose that \(G\) is at most \(n-1\) colorable. Then there exists an incidence map \(f\)
		\[B(G) \overset{f}{\rightarrow} B(K_{n-1}).\]
		Consider an injective vertex map \(i\),
		\[B(K_{n-1}) \overset{i}{\rightarrow}B(K_n).\]
		Then there is an edge \(e\) in \(K_n\) with no pre-image edge in \(K_{n-1}\) over \(i\). The assumption implies the the composition \(i f\) is surjective. So, there is an edge \(e'\) in \(G\) such that \(i f (e') = e\). Since \(f\) is an incidence map, \(f(ab)\) is an edge in \(K_{n-1}\) mapped to \(e\) over \(i\). This contradicts the definition of \(e\). Hence, \(n\) is the smallest integer for which \(G\) is \(n\)-colorable and therefore, \(n\) is the chromatic number of \(G\).
	\end{proof}
	
	\begin{remark}
		It follows immediately from Theorem \ref{ColoringintermsofVertexMaps} that
		\begin{equation}
			\label{Chromaticislessthantheta}
			\chi(G) \leq \theta(G).
		\end{equation}
		However, inequality \eqref{Chromaticislessthantheta} is not a good bound. In fact, given a natural number \(n \geq 3\), we can find a cycle \(C_m\) such that \(\theta (C_m) = n\) but \(\chi(C_m) \leq 3\). So \(\theta(G) - \chi(G)\) can be made arbitrarily large.
	\end{remark}

	\section{Conclusion}
	\label{Conclusion:sec}
	
	The star topology on the bare representations provide a novel framework for
	analyzing continuity and homomorphisms. While every homomorphism induces a
	continuous map, the converse fails due to potential loss of adjacency
	information. This distinction underscores the need for a nuanced understanding of
	the interplay between topological and combinatorial structures in graphs.
	
	\bibliographystyle{plain}
	\bibliography{References}
	
\end{document}